\documentclass[12pt]{amsart}

\usepackage{amsmath}
\usepackage{amssymb}
\usepackage{longtable}

\theoremstyle{plain}
\newtheorem{lemma}{Lemma}[section]
\newtheorem{proposition}[lemma]{Proposition}
\newtheorem{theorem}[lemma]{Theorem}
\newtheorem{corollary}[lemma]{Corollary}
\newtheorem{problem}[lemma]{Problem}

\theoremstyle{definition}

\newtheorem{remark}[lemma]{Remark}
\newtheorem{example}[lemma]{Example}

\title{Epimorphisms between $2$-bridge knot groups and their crossing numbers} 

\author{Masaaki Suzuki}

\address{Department of Frontier Media Science, Meiji University}

\email{macky@fms.meiji.ac.jp}

\begin{document}

\begin{abstract}
Suppose that 
there exists an epimorphism from the knot group of a $2$-bridge knot $K$ 
onto that of another knot $K'$. 
In this paper, 
we study the relationship between their crossing numbers $c(K)$ and $c(K')$. 
Especially it is shown that $c(K)$ is greater than or equal to $3 c(K')$ 
and we estimate how many knot groups a $2$-bridge knot group maps onto. 
Moreover, we formulate the generating function which 
determines the number of $2$-bridge knot groups admitting epimorphisms 
onto the knot group of a given $2$-bridge knot. 
\end{abstract}

\maketitle
\section{Introduction}\label{sect:intro}

Let $K$ be a knot and $G(K)$ the knot group, namely, 
the fundamental group of the exterior of $K$ in $S^3$. 
We denote by $c(K)$ the crossing number of $K$. 
Recently, many papers have studied epimorphisms between knot groups. 
One of the main interests of their papers was Simon's conjecture: 
every knot group maps onto at most finitely many knot groups. 
For example, Boileau, Boyer, Reid, and Wang in \cite{BBRW} showed that 
Simon's conjecture is true for $2$-bridge knots. 
Finally, Agol and Liu in \cite{agolliu} proved that Simon's conjecture holds for all knots. 

In \cite{kitano-suzuki1} and \cite{HKMS}, 
the existence and non-existence of 
a meridional epimorphism between knot groups of prime knots with up to $11$ crossings 
are determined completely. 
We say that a homomorphism from $G(K)$ to $G(K')$ is {\it meridional} 
if a meridian of $G(K)$ is sent to a meridian of $G(K')$, 
see also \cite{chasuzuki}. 
This result raises the following question: 
if there exists an epimorphism from $G(K)$ onto $G(K')$, 
then is $c(K)$ greater than or equal to $c(K')$?
This question is also mentioned in \cite{kitano-suzuki2}. 
If the answer is affirmative, then we obtain another proof for Simon's conjecture. 
This paper gives a partial affirmative answer for this question. 
That is to say, 
if there exists an epimorphism from the knot group of a $2$-bridge knot $K$ 
onto that of another knot $K'$, 
then $c(K)$ is greater than or equal to $3 c(K')$. 

In order to prove this result, we make use of the Ohtsuki-Riley-Sakuma construction. 
Ohtsuki, Riley, and Sakuma in \cite{ORS} established 
a systematic construction of epimorphisms between $2$-bridge knot groups. 
Besides, Garrabrant, Hoste, and Shanahan in \cite{GHS} gave 
necessary and sufficient conditions for any set of $2$-bridge knots 
to have an upper bound with respect to the Ohtsuki-Riley-Sakuma construction. 
Conversely, it is shown 
that all epimorphisms between $2$-bridge knot groups arise from the Ohtsuki-Riley-Sakuma construction, 
as a consequence of Agol's result announced in \cite{agol}. 
Aimi, Lee, and Sakuma give another proof in \cite{ALS} for this result.

In this paper, we consider the crossing numbers of $2$-bridge knots 
whose knot groups admit epimorphisms onto a $2$-bridge knot group. 
By using this result, 
we estimate how many knot groups a $2$-bridge knot group maps onto. 
Furthermore, 
we formulate the generating function 
which determines the number of $2$-bridge knots $K$ 
admitting epimorphisms from $G(K)$ onto the knot group of a given $2$-bridge knot. 

Throughout this paper, we do not distinguish a knot from its mirror image, 
since their knot groups are isomorphic and we discuss epimorphisms between knot groups. 
The numberings of the knots with up to $10$ and $11$ crossings follow
Rolfsen's book \cite{rolfsen} and 
the web page KnotInfo \cite{CL} by Cha and Livingston, respectively

\section{$2$-bridge knot and continued fraction expansion}\label{sect:cfe}

In this section, we recall some well-known results on $2$-bridge knots. 
See \cite{bzh}, \cite{murasugi} in detail, for example. 

A $2$-bridge knot corresponds to a rational number $r = q/p \in {\mathbb Q}$. 
Then we denote by $K(q/p)$ such a $2$-bridge knot. 
Schubert classified $2$-bridge knots as follows. 

\begin{theorem}[Schubert] 
Let $K(q/p)$ and $K(q'/p')$ be $2$-bridge knots. 
These knots are equivalent if and only if 
the following conditions hold. 
\begin{enumerate}
\item $p = p'$. 
\item Either $q \equiv \pm q' \pmod{p}$ or $q q' \equiv \pm 1 \pmod{p}$. 
\end{enumerate}
\end{theorem}

By using this theorem, it is sufficient to consider $r \in {\mathbb Q} \cap (0,\frac{1}{2})$. 
Note that $K(0)$ is the trivial link and that $K(1/2)$ is the Hopf link. 
A rational number $q/p \in {\mathbb Q} \cap (0 , \frac{1}{2})$ can 
be expressed as a continued fraction expansion:
\[
 \frac{q}{p} = [a_1,a_2, \ldots, a_{m-1},a_m] = 
\frac{1}{a_1 + \frac{1}{a_2 + \frac{1}{\ddots \frac{1}{a_{m-1} + \frac{1}{a_m}}}}} ,
\]
where $a_1 > 0$. 
Note that a rational number admits many continued fraction expansions. 
For example, we have $29/81 = [3,-5,4,1,-2] = [2,1,3,1,5]$. 
It is easy to see that the following properties are satisfied. 
First, we can delete zeros in a continued fraction expansion by using 
\[
 [a_1,a_2,\ldots,a_{i-2}, a_{i-1},0,a_{i+1},a_{i+2}, \ldots, a_m] 
= [a_1,a_2,\ldots,a_{i-2}, a_{i-1} + a_{i+1}, a_{i+2},\ldots, a_m] .
\]
If we consider a $2$-bridge knot, we may assume that $a_1, a_m \neq \pm 1$, 
since 
\[
 [a_1,a_2,\ldots, a_{m-1},\pm 1] = [a_1,a_2,\ldots, a_{m-1} \pm 1] 
\]
and $K([a_1,a_2,\ldots,a_m])$ is equivalent to $K([a_{m},a_{m-1},\ldots,a_1])$ 
up to mirror image. 
Moreover, the Euclidean algorithm allows us to take 
a continued fraction expansion such that 
all $a_i$ in $[a_1,a_2,\ldots,a_m]$ are positive.  

If a rational number $r$ is expressed as $[a_1,a_2,\ldots,a_m]$ where $a_i > 0$ and $a_1,a_m \geq 2$, 
then the continued fraction expansion is called {\it standard}. 
By the above arguments, 
we can always take the standard continued fraction expansion 
of the rational number $r$ for a $2$-bridge knot $K(r)$. 
Furthermore the standard continued fraction expansion gives us 
the unique continued fraction expansion of the rational number which corresponds to a $2$-bridge knot 
in the following sense. 
Let $K(q/p)$ and $K(q'/p')$ be $2$-bridge knots. 
Suppose that these rational numbers are written as the standard continued fraction expansions 
$q/p = [a_1,a_2,\ldots,a_m], q'/p' = [a_1',a_2',\ldots,a_{m'}']$. 
It is known that 
$K(q/p)$ and $K(q'/p')$ are equivalent up to mirror image if and only if 
\[
 (a_1,a_2,\ldots,a_m) = (a_1',a_2',\ldots,a_{m'}') \mbox{ or } (a_{m'}',a_{m'-1}',\ldots,a_1'). 
\]

Thistlethwaite \cite{thi}, Kauffman \cite{kauffman} and Murasugi \cite{murasugi1},\cite{murasugi2} independently 
proved the first Tait conjecture. 
Hence, we can determine the crossing number of a $2$-bridge knot 
by using the standard continued fraction expansion. 
Namely, the crossing number for the standard continued fraction $[a_1,a_2,\ldots, a_m]$ 
is given by 
\[
 c(K([a_1,a_2,\ldots,a_m])) = \sum_{i=1}^m a_i .
\]

\section{Epimorphisms between $2$-bridge knot groups}

We have the following remarkable result about epimorphisms between $2$-bridge knot groups. 
An epimorphism between $2$-bridge knot groups is always meridional. 
Moreover, the rational numbers for these $2$-bridge knots have the following relationship. 

\begin{theorem}[Ohtsuki-Riley-Sakuma \cite{ORS}, Agol \cite{agol}, Aimi-Lee-Sakuma \cite{ALS}]\label{thm:ors}
Let $K(r), K(\tilde{r})$ be $2$-bridge knots, 
where $r = [a_1,a_2,\ldots,a_m]$. 
If there exists an epimorphism $\varphi : G(K(\tilde{r})) \to G(K(r))$, 
then $\varphi$ is meridional and $\tilde{r}$ can be written as 
\[
 \tilde{r} = 
[\varepsilon_1 {\bf a}, 2 c_1, 
\varepsilon_2 {\bf a}^{-1}, 2 c_2, 
\varepsilon_3 {\bf a}, 2 c_3, 
\varepsilon_4 {\bf a}^{-1}, 2 c_4, 
\ldots, 
\varepsilon_{2n} {\bf a}^{-1}, 2 c_{2n}, \varepsilon_{2n+1} {\bf a}] , 
\]
where ${\bf a} = (a_1, a_2,\ldots,a_m), {\bf a}^{-1} = (a_m, a_{m-1},\ldots,a_1)$, 
$\varepsilon_i = \pm 1 \, \, (\varepsilon_1 = 1)$, and $c_i \in {\mathbb Z}$. 
\end{theorem}

\begin{remark}
If a rational number $\tilde{r}$ is expressed as 
\[
 \tilde{r} = 
[\varepsilon_1 {\bf a}, 2 c_1, 
\varepsilon_2 {\bf a}^{-1}, 2 c_2, 
\varepsilon_3 {\bf a}, 2 c_3, 
\varepsilon_4 {\bf a}^{-1}, 2 c_4, 
\ldots, 
\varepsilon_{2n} {\bf a}^{-1}, 2 c_{2n}, \varepsilon_{2n+1} {\bf a}] , 
\]
then we say that 
$\tilde{r}$ has an expansion of {\it type} $2n+1$ with respect to ${\bf a} = (a_1, a_2,\ldots,a_m)$. 
\begin{enumerate}
\item 
In this paper, we do not need to consider an expression of type $2n$ with respect to ${\bf a}$, 
since $K([\varepsilon_1 {\bf a},2 c_1, \ldots, 2 c_{2n-1}, \varepsilon_{2n} {\bf a}^{-1}])$ 
is a $2$-bridge link. 
\item If $c_i=0$ and $\varepsilon_i \cdot \varepsilon_{i+1}=-1$, then 
\begin{align*}
\tilde{r}&=[\ldots,\, \varepsilon_{i-1} {\bf a}^{\pm1} ,\, 2 c_{i-1},\, \varepsilon_{i} {\bf a}^{\mp1} ,\, 0,\, 
 \varepsilon_{i+1} {\bf a}^{\pm1} ,\, 2 c_{i+1}, \varepsilon_{i+2} {\bf a}^{\mp1} ,\, \ldots ] \\
&=[\ldots,\varepsilon_{i-1} {\bf a}^{\pm1} ,\, 2 c_{i-1},\, 0,\, 
 2 c_{i+1},\, \varepsilon_{i+2} {\bf a}^{\mp1} ,\, \ldots ] \\
&=[\ldots,\, \varepsilon_{i-1} {\bf a}^{\pm1} ,\, 2 (c_{i-1} + c_{i+1}),\, \varepsilon_{i+2} {\bf a}^{\mp1} ,\, \ldots ] .
\end{align*}
It follows that $\tilde{r}$ has type $2n-1$. 
Hence we do not deal with the case $c_i=0$, $\varepsilon_i \cdot \varepsilon_{i+1}=-1$. 
\end{enumerate}
\end{remark}

\begin{example}\label{ex:ORS} 
For example, we consider a $2$-bridge knot $K(5/27)$. 
The rational number $5/27$ has continued fraction expansions: 
\[
 \frac{5}{27} = [5,2,2] = [3,0,3,-2,3].
\]
The second expression implies that the crossing number of $K(5/27)$ is $9$. 
The last expression is of type $3$ with respect to ${\bf a} = (3)$. 
Therefore the knot group $G(K(5/27))$ 
admits an epimorphism onto the trefoil knot group $G(3_1) = G(K(1/3)) = G(K([3]))$. 
Similarly, we have 
\[
 \frac{1}{9} = [9] = [3,0,3,0,3], \quad 
 \frac{19}{45} = [2,2,1,2,2] = [3,-2,3,-2,3] . 
\]
It follows that there exist epimorphisms from $G(K(1/9))$ and $G(K(19/45))$ 
onto the trefoil knot group. 

The previous papers \cite{kitano-suzuki1} and \cite{HKMS} determined 
all the pairs of prime knots with up to $11$ crossings 
which admit meridional epimorphisms between their knot groups. 
The results in \cite{kitano-suzuki1}, \cite{HKMS} coincide with the above examples. 
Note that $K(1/9) = 9_1$, $K(5/27) = 9_6$, and $K(19/45) = 9_{23}$. 
\end{example}

In general, even if $[a_1,\ldots,a_m]$ is the standard continued fraction expansion, 
and $\tilde{r}$ is of type $2n+1$ with respect to $(a_1,\ldots,a_m)$, 
then this expansion of $\tilde{r}$ may not be standard. 
However, we can get the standard continued fraction expansion 
and then determine the crossing number of $K(\tilde{r})$. 

\begin{theorem}\label{thm:crossingnumber}
Let $[a_1,\ldots,a_m]$ be the standard continued fraction expansion. 
Suppose that a rational number $\tilde{r}$ has a continued fraction expansion 
of type $2n+1$ with respect to ${\bf a} = (a_1,\ldots,a_m)$: 
\[
 \tilde{r} = 
[\varepsilon_1 {\bf a}, 2 c_1, 
\varepsilon_2 {\bf a}^{-1}, 2 c_2, 
\varepsilon_3 {\bf a}, 2 c_3, 
\varepsilon_4 {\bf a}^{-1}, 2 c_4, 
\ldots, 
\varepsilon_{2n} {\bf a}^{-1}, 2 c_{2n}, \varepsilon_{2n+1} {\bf a}] . 
\]
Then the crossing number of $K(\tilde{r})$ is given by 
\[
 c(K(\tilde{r})) = 
(2n + 1) |{\bf a}| + \sum_{i=1}^{2n} (2 |c_i| - \psi (i) - \bar{\psi} (i)), 
\]
where $|{\bf a}| = \sum_{i=1}^m a_i$ and 
\[
\psi (i) = 
\left\{
\begin{array}{ll}
1 & \, \, \varepsilon_i \cdot c_i < 0 \\
0 & \, \, \varepsilon_i \cdot c_i \geq 0 
\end{array}
\right. , \qquad
\bar{\psi} (i) = 
\left\{
\begin{array}{ll}
1 & \, \, c_i \cdot \varepsilon_{i+1} < 0 \\
0 & \, \, c_i \cdot \varepsilon_{i+1} \geq 0 
\end{array}
\right. .
\]
\end{theorem}

Remark that 
\[
 \sum_{i=1}^{2n} (\psi (i) + \bar{\psi} (i)) 
\]
is the number of sign changes. 
In order to prove Theorem \ref{thm:crossingnumber}, 
we prepare the following lemma. 
Namely, negative integers in a continued fraction expansion 
can be changed into positive integers. 

\begin{lemma}\label{lem:nonpositivecfe}
Let $a_1,\ldots, a_k, b_1, \ldots, b_l, c_1, \ldots, c_m$ be integers. 
Then we have 
\begin{enumerate}
\item case $l \geq 2$: 
\begin{align*}
& [a_1, \ldots, a_k, - b_1, -b_2,\ldots,-b_{l-1}, - b_l,c_1,\ldots, c_m] = \\
& [a_1, \ldots, ,a_{k-1},a_k-1,1, b_1-1, b_2, \ldots,b_{l-1}, b_l-1,1,c_1-1,c_2,\ldots, c_m], 
\end{align*}
\item case $l =1$ and $b_1 \geq 2$: 
\[
 [a_1, \ldots, a_k, - b_1,c_1,\ldots, c_m] = 
 [a_1, \ldots, ,a_{k-1},a_k-1,1, b_1-2,1,c_1-1,c_2,\ldots, c_m], 
\]
\item case $l \geq 2$: 
\[
 [a_1, \ldots, a_k, - b_1, \ldots, - b_l] = 
 [a_1, \ldots, ,a_{k-1},a_k-1,1, b_1-1, b_2, \ldots, b_l],  
\]
\item case $l =1$ and $b_1 \geq 2$: 
\[
 [a_1, \ldots, a_k, - b_1] = 
 [a_1, \ldots, ,a_{k-1},a_k-1,1, b_1-1]. 
\]
\end{enumerate}
\end{lemma}

\begin{proof}
Recall the matrix representation of a continued fraction expansion (see for instance \cite{thomson}). 
For a continued fraction $[x_1,x_2,\ldots,x_m]$, we define $p,q$ by 
\[
\left(
\begin{array}{cc}
x_1 & 1 \\
1 & 0 
\end{array}
\right)
\left(
\begin{array}{cc}
x_2 & 1 \\
1 & 0 
\end{array}
\right)
\cdots
\left(
\begin{array}{cc}
x_m & 1 \\
1 & 0 
\end{array}
\right)
= 
\left(
\begin{array}{cc}
p & * \\
q & * 
\end{array}
\right). 
\]
It is known that we have an equality 
\[
 [x_1,x_2,\ldots,x_m] = \frac{q}{p}.
\]

We will show $(1)$ by using 
the above matrix representations of the both sides of $(1)$. 
Let $A,C$ be the matrices defined by 
\[
A = 
\left(
\begin{array}{cc}
a_1 & 1 \\
1 & 0 
\end{array}
\right)
\cdots
\left(
\begin{array}{cc}
a_{k-1} & 1 \\
1 & 0 
\end{array}
\right), \qquad 
C = 
\left(
\begin{array}{cc}
c_2 & 1 \\
1 & 0 
\end{array}
\right)
\cdots
\left(
\begin{array}{cc}
c_m & 1 \\
1 & 0 
\end{array}
\right)
\]
respectively and we set $B$ by 
\[
B = 
\left(
\begin{array}{cc}
B_{11} & B_{12} \\
B_{21} & B_{22} 
\end{array}
\right) = 
\left(
\begin{array}{cc}
b_2 & 1 \\
1 & 0 
\end{array}
\right)
\cdots
\left(
\begin{array}{cc}
b_{l-1} & 1 \\
1 & 0 
\end{array}
\right).
\]
The matrix representation of the left hand side of $(1)$ is 
\begin{align*}
& 
\left(
\begin{array}{cc}
a_1 & 1 \\
1 & 0 
\end{array}
\right)
\cdots
\left(
\begin{array}{cc}
a_{k-1} & 1 \\
1 & 0 
\end{array}
\right)
\left(
\begin{array}{cc}
a_k & 1 \\
1 & 0 
\end{array}
\right)
\left(
\begin{array}{cc}
-b_1 & 1 \\
1 & 0 
\end{array}
\right)
\left(
\begin{array}{cc}
-b_2 & 1 \\
1 & 0 
\end{array}
\right)
\cdots \\
& \qquad \qquad 
\left(
\begin{array}{cc}
-b_{l-1} & 1 \\
1 & 0 
\end{array}
\right)
\left(
\begin{array}{cc}
-b_l & 1 \\
1 & 0 
\end{array}
\right)
\left(
\begin{array}{cc}
c_1 & 1 \\
1 & 0 
\end{array}
\right)
\left(
\begin{array}{cc}
c_2 & 1 \\
1 & 0 
\end{array}
\right)
\cdots
\left(
\begin{array}{cc}
c_m & 1 \\
1 & 0 
\end{array}
\right) \\
& = 
A
\left(
\begin{array}{cc}
- a_k b_1 + 1& a_k \\
-b_1 & 1 
\end{array}
\right)
\left(
\begin{array}{cc}
(-1)^{l} B_{11} & (-1)^{l+1} B_{12} \\
(-1)^{l+1} B_{21} & (-1)^{l} B_{22} 
\end{array}
\right)
\left(
\begin{array}{cc}
-b_l c_1 + 1 & -b_l \\
c_1 & 1 
\end{array}
\right)
C \\
& = 
A
\left(
\begin{array}{cc}
a_k b_1 - 1& a_k \\
b_1 & 1 
\end{array}
\right)
\left(
\begin{array}{cc}
(-1)^{l} B_{11} & (-1)^{l} B_{12} \\
(-1)^{l} B_{21} & (-1)^{l} B_{22} 
\end{array}
\right)
\left(
\begin{array}{cc}
b_l c_1 - 1 & b_l \\
c_1 & 1 
\end{array}
\right)
C \\
& = 
(-1)^{l} 
A
\left(
\begin{array}{cc}
a_k b_1 - 1& a_k \\
b_1 & 1 
\end{array}
\right)
\left(
\begin{array}{cc}
B_{11} & B_{12} \\
B_{21} & B_{22} 
\end{array}
\right)
\left(
\begin{array}{cc}
b_l c_1 - 1 & b_l \\
c_1 & 1 
\end{array}
\right)
C \\
& = 
(-1)^{l} 
\left(
\begin{array}{cc}
a_1 & 1 \\
1 & 0 
\end{array}
\right)
\cdots
\left(
\begin{array}{cc}
a_{k-1} & 1 \\
1 & 0 
\end{array}
\right)
\left(
\begin{array}{cc}
a_k - 1& 1 \\
1 & 0 
\end{array}
\right)
\left(
\begin{array}{cc}
1& 1 \\
1 & 0 
\end{array}
\right)
\left(
\begin{array}{cc}
b_1- 1 & 1 \\
1 & 0 
\end{array}
\right)
\left(
\begin{array}{cc}
b_2 & 1 \\
1 & 0 
\end{array}
\right)
\cdots \\
& \qquad \qquad 
\left(
\begin{array}{cc}
b_{l-1} & 1 \\
1 & 0 
\end{array}
\right)
\left(
\begin{array}{cc}
b_l -1& 1 \\
1 & 0 
\end{array}
\right)
\left(
\begin{array}{cc}
1 & 1 \\
1 & 0 
\end{array}
\right)
\left(
\begin{array}{cc}
c_1 - 1& 1 \\
1 & 0 
\end{array}
\right)
\left(
\begin{array}{cc}
c_2 & 1 \\
1 & 0 
\end{array}
\right)
\cdots
\left(
\begin{array}{cc}
c_m & 1 \\
1 & 0 
\end{array}
\right) .
\end{align*}
The last expression is $(-1)^l$ times the matrix representation of the right hand side of $(1)$. 
Therefore the rational numbers of the both sides of $(1)$ coincide. 

Next, we see the matrix representation of the left hand side of $(2)$: 
\begin{align*}
& 
\left(
\begin{array}{cc}
a_1 & 1 \\
1 & 0 
\end{array}
\right)
\cdots
\left(
\begin{array}{cc}
a_{k-1} & 1 \\
1 & 0 
\end{array}
\right)
\left(
\begin{array}{cc}
a_k & 1 \\
1 & 0 
\end{array}
\right)
\left(
\begin{array}{cc}
-b_1 & 1 \\
1 & 0 
\end{array}
\right)
\left(
\begin{array}{cc}
c_1 & 1 \\
1 & 0 
\end{array}
\right)
\left(
\begin{array}{cc}
c_2 & 1 \\
1 & 0 
\end{array}
\right)
\cdots
\left(
\begin{array}{cc}
c_m & 1 \\
1 & 0 
\end{array}
\right) \\
& = 
A
\left(
\begin{array}{cc}
- a_k b_1 c_1 + a_k + c_1 & - a_k b_1 + 1\\
-b_1 c_1 + 1 & -b_1 
\end{array}
\right)
C \\
& = 
(-1) A
\left(
\begin{array}{cc}
a_k b_1 c_1 - a_k - c_1 & a_k b_1 - 1\\
b_1 c_1 - 1 & b_1 
\end{array}
\right)
C \\
& = 
(-1) 
\left(
\begin{array}{cc}
a_1 & 1 \\
1 & 0 
\end{array}
\right)
\cdots
\left(
\begin{array}{cc}
a_{k-1} & 1 \\
1 & 0 
\end{array}
\right)
\left(
\begin{array}{cc}
a_k - 1 & 1 \\
1 & 0 
\end{array}
\right)
\left(
\begin{array}{cc}
1 & 1 \\
1 & 0 
\end{array}
\right)
\left(
\begin{array}{cc}
b_1 - 2& 1 \\
1 & 0 
\end{array}
\right) \\
& \qquad \qquad \qquad 
\left(
\begin{array}{cc}
1 & 1 \\
1 & 0 
\end{array}
\right)
\left(
\begin{array}{cc}
c_1 - 1 & 1 \\
1 & 0 
\end{array}
\right)
\left(
\begin{array}{cc}
c_2 & 1 \\
1 & 0 
\end{array}
\right)
\cdots
\left(
\begin{array}{cc}
c_m & 1 \\
1 & 0 
\end{array}
\right) .
\end{align*}
The last expression is also $(-1)$ times the matrix representation of the right hand side of $(2)$. 
Hence these continued fraction expressions represent the same rational number. 

The similar proof works for $(3)$ and $(4)$. 
\end{proof}

\begin{example}
Suppose that a rational number $29/81$ is expressed as $[3,-5,4,1,-2]$. 
The above arguments show that 
\[
[3,-5,4,1,-2] = [2,1,3,1,3,0,1,1] = [2,1,3,1,4,1] = [2,1,3,1,5] .
\]
The last expression is the standard continued fraction expansion 
and then the crossing number of $K(29/81)$ is 
\[
 c(K(29/81)) = c(K([3,-5,4,1,-2])) = c(K([2,1,3,1,5])) = 12. 
\]
\end{example}

In Lemma \ref{lem:nonpositivecfe}, 
if all $a_i, b_i, c_i$ are positive, 
then each integer of the right hand sides is positive or zero. 
Hence, we can obtain the standard continued fraction expansion and 
determine the crossing number of a $2$-bridge knot. 

\begin{corollary}\label{cor:36}
Let $a_i,b_i,c_i$ be positive integers. 
If $l \neq 1$ or $b_1 \geq 2$, then we have 
\begin{enumerate}
\item 
$\displaystyle{c(K([a_1, \ldots, a_k, - b_1,\ldots, - b_l,c_1,\ldots, c_m])) = 
\sum_{i=1}^k a_i + \sum_{i=1}^l b_i + \sum_{i=1}^m c_i - 2}$,
\item 
$\displaystyle{c(K([a_1, \ldots, a_k, - b_1,\ldots, - b_l])) = 
\sum_{i=1}^k a_i + \sum_{i=1}^l b_i - 1}$.
\end{enumerate}
\end{corollary}

Corollary \ref{cor:36} suggests how to determine the crossing number 
without using the explicit standard continued fraction expansion. 
To be precise, 
it is sufficient to compute the sum of the absolute values in a continued fraction expansion 
and to count the number of sign changes. 
In the above example, 
the signs of components in $[3,-5,4,1,-2]$ are changed $3$ times. 
Then the crossing number is 
\[
 c(K([3,-5,4,1,-2])) = |3| + |-5| + |4| + |1| + |-2| - 3 = 12. 
\]
These arguments show Theorem \ref{thm:crossingnumber}.

\begin{proof}[Proof of Theorem \ref{thm:crossingnumber}]
The sum of the absolute values of components in $\tilde{r}$ is 
\[
  (2n + 1) |{\bf a}| + \sum_{i=1}^{2n} (2 |c_i|) .  
\]
By Lemma \ref{lem:nonpositivecfe}, 
if the signs in a continued fraction expansion of $\tilde{r}$ are changed $k$ times, 
then the crossing number of $K(\tilde{r})$ is decreased by $k$, from the above value. 
The number of sign changes in $\tilde{r}$ is 
\[
 \sum_{i=1}^{2n} (\psi (i) + \bar{\psi} (i)) 
\]
by definition. 
Since $\tilde{r}$ is an expression of type $2n+1$ with respect to standard ${\bf a}$, 
we can apply Corollary \ref{cor:36}. 
Therefore this completes the proof.
\end{proof}

We define 
$\bar{c}_i$ by $2 |c_i| - \psi (i) - \bar{\psi} (i)$. 
Then $\bar{c}_i$ is not negative. 

\begin{proposition}\label{prop:ci}
Suppose that $\tilde{r}$ is as above. 
Then we have $\bar{c}_i \geq 0 $ for any $i \,\,(1 \leq i \leq 2n)$. 
\end{proposition}

\begin{proof}
If $c_i \neq 0$, then $2 |c_i| \geq 2$. 
On the other hand, $\psi (i)$ and $\bar{\psi} (i)$ are $0$ or $1$, 
and then we get $\bar{c}_i \geq 0 $. 
If $c_i = 0$, then $\psi (i) = 0$ and $\bar{\psi} (i) = 0$ by definition. 
Therefore $\bar{c}_i = 0$ in this case. 
\end{proof}

\section{Simon's conjecture}

Simon's conjecture for $2$-bridge knots is proved in \cite{BBRW}, 
and for all knots in \cite{agolliu}, as mentioned in Section \ref{sect:intro}. 
In this section, we investigate how many knot groups a $2$-bridge knot group maps onto. 
 
Let $EK(n)$ be the maximal number of knots 
whose knot groups a $2$-bridge knot group with $n$ crossings admits epimorphisms onto.  
Theorem \ref{thm:crossingnumber} and Proposition \ref{prop:ci} imply the following, 
which is one of the main results in this paper. 
It gives us a rough estimation of $EK(n)$. 

\begin{theorem}\label{cor:39}
Let $K(\tilde{r})$ be a $2$-bridge knot. 
If there exists an epimorphism from $G(K(\tilde{r}))$ onto the knot group of another knot $K$, 
then 
\[
 c(K(\tilde{r})) \geq 3 \, c(K). 
\]
In particular, 
all the $2$-bridge knots $K$ with up to $8$ crossings are minimal,
that is to say, if $G(K)$ admits an epimorphism onto a knot group $G(K')$, 
then $K'$ is equivalent to $K$ or the trivial knot. 
\end{theorem}

\begin{proof}
By \cite{BBRW}, \cite{silverwhitten}, 
if $G(K(\tilde{r}))$ admits an epimorphism onto $G(K)$, 
then $K$ is also a $2$-bridge knot or the trivial knot. 
In the case that $K$ is the trivial knot, the inequality, which we show, holds obviously. 
Next, we assume that $K$ is a $2$-bridge knot and 
that $r$ is the corresponding rational number.
Take the standard continued fraction expansion $[a_1,a_2,\ldots, a_m]$ of $r$, 
then $\tilde{r}$ has an expansion of type $2n+1$ with respect to ${\bf a} = (a_1,a_2,\ldots, a_m)$. 
By Theorem \ref{thm:crossingnumber}, we have 
\begin{align*}
c(K(\tilde{r})) 
&= 
(2n + 1) |{\bf a}| + \sum_{i=1}^{2n} (2 |c_i| - \psi (i) - \bar{\psi} (i)) \\ 
&= (2n + 1) \, c(K) + \sum_{i=1}^{2n} \bar{c}_i \\
&\geq (2n + 1) \, c(K) \qquad \mbox{by Proposition \ref{prop:ci}} \\
&\geq 3 \, c(K). 
\end{align*}
Furthermore, since a non-trivial knot has at least $3$ crossings, 
all the $2$-bridge knots with up to $8$ crossings are minimal. 
\end{proof}

\begin{remark}
The previous paper \cite{kitano-suzuki1} shows that 
the knot groups of $7$ knots with less than $9$ crossings 
admit epimorphisms onto the trefoil knot group. 
To be precise, they are $3$-bridge knots 
$8_5,8_{10},8_{15},8_{18},8_{19},8_{20},8_{21}$. 
Then this inequality does not hold for $3$-bridge knots. 
\end{remark}

Ernst and Sumners in \cite{ernstsumners} determined the number $TK(n)$ of $2$-bridge knots 
in terms of the crossing number $n \geq 3$ as follows: 
\[
TK(n) =
\left\{
\begin{array}{ll}
\frac{1}{3} \left( 2^{(n-3)} + 2^{(n-4)/2} \right) & n \equiv 0 \pmod 4 \\
\frac{1}{3} \left( 2^{(n-3)} + 2^{(n-3)/2} \right) & n \equiv 1 \pmod 4 \\
\frac{1}{3} \left( 2^{(n-3)} + 2^{(n-4)/2} - 1 \right) \quad & n \equiv 2 \pmod 4 \\
\frac{1}{3} \left( 2^{(n-3)} + 2^{(n-3)/2} + 1 \right) & n \equiv 3 \pmod 4 
\end{array}
\right. .
\]
Then we can estimate $EK(n)$ by using Theorem \ref{cor:39}: 
\[
EK(n) \leq \sum_{k=3}^{\lfloor n/3 \rfloor} TK(k) .
\]
These numbers are obtained as shown in Table \ref{tbl:roughestimation}. 
\begin{table}[h]
\caption{$\sum_{k=3}^{\lfloor n/3 \rfloor} TK(k)$}\label{tbl:roughestimation}
\begin{center}
\begin{tabular}{|c|r|r|r|r|r|r|r|r|}
\hline 
$n$ & $9,10,11$ & $12,13,14$ & $15,16,17$ & $18,19,20$ & $21,22,23$ & $24,25,26$ 
& $27,28,29$ & $30,31,32$ \\
\hline
& $1$ & $2$ & $4$ & $7$ & $14$ & $26$ & $50$ & $95$ \\
\hline
\end{tabular}
\end{center}
\end{table}

In particular, we obtain that the knot groups of $2$-bridge knots with $12,13,$ and $14$ crossings 
map onto at most two knot groups, 
which are the trefoil knot group $G(3_1)$ and the figure eight knot group $G(4_1)$. 
On the other hand, 
Garrabrant, Hoste, and Shanahan studied an upper bound for a set of $2$-bridge knots 
with respect to epimorphisms between their knot groups. 
We recall Garrabrant-Hoste-Shanahan's arguments more precisely. 
Let ${\bf a} = (a_1,a_2,\ldots,a_{2n})$ be a vector such that 
\begin{itemize}
\item[(1)] each $a_i \in \{-2,0,2\}$, 
\item[(2)] $a_1 \neq 0$ and $a_{2n} \neq 0$, 
\item[(3)] if $a_i = 0$, then $a_{i-1} = a_{i+1} = \pm 2$.
\end{itemize}
We call such a $[a_1,a_2,\ldots,a_{2n}]$ an {\it even standard continued fraction expansion}. 
Remark that if we consider ${\bf a} = (a_1,a_2,\ldots,a_{2n})$ up to equivalent relations 
${\bf a} = \pm {\bf b}$ and ${\bf a} = \pm {\bf b}^{-1}$ 
where ${\bf b}^{-1}$ is ${\bf b}$ read backwards, 
then a $2$-bridge knot can be expressed uniquely as $K([a_1,a_2,\ldots,a_{2n}])$ 
by using an even standard continued fraction expansion. 

\begin{proposition}[Garrabrant-Hoste-Shanahan \cite{GHS}]\label{prop:ghs}
Let $[a_1,a_2,\ldots,a_{2n}]$ and $[b_1,b_2,\ldots,b_{2n}]$ be 
even standard continued fraction expansions of the same length. 
If a $2$-bridge knot group admits epimorphisms 
onto $G(K([a_1,a_2,\ldots, a_{2n}]))$ and $G(K([b_1,b_2,\ldots, b_{2n}]))$, 
then $(a_1,a_2,\ldots,a_{2n}) = (b_1,b_2,\ldots,b_{2n})$. 
\end{proposition}

For example, the trefoil is $3_1 = K([2,-2])$ and the figure eight knot is $4_1 = K([2,2])$. 
Since the lengths of these even standard continued fraction expansions are the same, 
there does not exist a $2$-bridge knot whose knot group admits 
epimorphisms onto $G(3_1)$ and $G(4_1)$ simultaneously, by Proposition \ref{prop:ghs}. 
Similarly, a $2$-bridge knot group maps onto 
the knot group of at most only one of $\{5_1, 5_2, 6_1, 6_2, 6_3\}$, 
since 
\[
 5_1 = K([2,-2,2,-2]), ~ 5_2 = K([2,-2,0,-2]), 
\]
\[
6_1 = K([2,0,2,2]), ~ 6_2 = K([2,2,-2,2]), ~ 6_3 = K([2,-2,-2,2]).
\] 

In order to extend this argument, we consider the relationship 
between the length of an even standard continued fraction expansion and 
the crossing number. 

\begin{proposition}\label{prop:ineq}
Let $[a_1,a_2,\ldots,a_{2n}]$ be an even standard continued fraction expansion. 
Then the crossing number $c(K([a_1,a_2,\ldots,a_{2n}]))$ satisfies the following inequalities: 
\[
 2n+1 \leq c(K([a_1,a_2,\ldots,a_{2n}])) \leq 4n .
\]
\end{proposition}

\begin{proof} 
First of all, we delete zeros in $(a_1,a_2,\ldots,a_{2n})$ as before: 
\[
 [a_1,a_2,\ldots,a_{2n}] = [a_1',a_2',\ldots,a_{2n'}'] ,
\]
where $a_i' \in 2 {\mathbb Z} \setminus \{ 0 \}$. 
Let $\ell$ be the number of zeros in $(a_1,a_2,\ldots,a_{2n})$. 
Then we have $2 \ell = 2 n - 2 n'$ and 
\[
 \sum_{i=1}^{2n'} \left( | a_i' | - 2 \right) = 2 \ell .
\]
It follows that 
\[
\sum_{i=1}^{2n'} | a_i' | = 2 \ell + 4 n' = 2 n + 2 n' . 
\]
By the same argument of Proof of Theorem \ref{thm:crossingnumber}, 
we obtain 
\begin{align*}
c(K([a_1,a_2,\ldots,a_{2n}])) 
= c(K([a_1',a_2',\ldots,a_{2n'}'])) 
= \sum_{i=1}^{2n'} | a_i' | - k 
= \sum_{i=1}^{2n} | a_i | - k , 
\end{align*}
where $k$ is the number of sign changes in $(a_1',a_2',\ldots,a_{2n'}')$. 
Note that $0 \leq k \leq 2 n' - 1$. 
(If all $a_i'$ are positive or negative, then $k=0$. 
If $a_i' \cdot a_{i+1}'< 0 $ for all $i ~ (0 \leq i \leq 2n'-1)$, then $k=2n'-1$.) 
By $|a_i| \leq 2$, 
\[
 \sum_{i=1}^{2n} | a_i | - k \leq 4n .
\]
Moreover, we obtain
\begin{align*}
\sum_{i=1}^{2n'} | a_i' | - k 
&= 2 n + 2 n' - k \\ 
& \geq 2 n + 2 n' - (2n'-1) \\
&= 2n+1 .
\end{align*}
This completes the proof. 
\end{proof}

By Proposition \ref{prop:ghs}, 
if two distinct $2$-bridge knots $K,K'$ have 
even standard continued fraction expansions of the same length, 
then there does not exist a $2$-bridge knot whose knot group maps onto $G(K)$ and $G(K')$. 
Combined with Proposition \ref{prop:ghs} and Proposition \ref{prop:ineq}, 
we can estimate $EK(n)$ more strictly.  

\begin{theorem}\label{thm:simons}
The number $EK(n)$ satisfies 
\[
 EK(n) \leq \left\lfloor \frac{n-3}{6} \right\rfloor .
\]
\end{theorem}

\begin{proof}
Let $K$ be a $2$-bridge knot with $n$ crossings. 
If $G(K)$ admits an epimorphism onto $G(K')$, 
then the crossing number of $K'$ is at most $\lfloor \frac{n}{3} \rfloor$, by Theorem \ref{cor:39}. 
Let $[a_1,a_2,\ldots,a_{2m}]$ be the even standard continued fraction expansion of $K'$, namely 
$K' = K([a_1,a_2,\ldots,a_{2m}])$. 
By Proposition \ref{prop:ghs}, 
$EK(n)$ is less than or equal to the number of the lengths 
of even standard continued fraction expansions. 
By Proposition \ref{prop:ineq}, we have 
\[
2m \leq \left\lfloor \frac{n}{3} \right\rfloor - 1. 
\]
Hence we obtain 
\[
 EK(n) \leq \left\lfloor \frac{\lfloor \frac{n}{3} \rfloor - 1}{2} \right\rfloor 
= \left\lfloor \frac{n-3}{6} \right\rfloor .
\]
\end{proof}

For example, the knot group of a $2$-bridge knot with $50$ crossings 
maps onto at most $7$ distinct knot groups. 
Actually, we can get the precise number $EK(n)$ for $n \leq 30$ by computer program: 
\[
 EK(n) = 
\left\{
\begin{array}{ll}
0 \quad & n = 3,4,5,6,7,8 \\
1 & n = 9,10,11,12,13,14,18,19,20,24 \\
2 & n = 15,16, 17, 21,22,23,25,26,27,28,29,30 
\end{array}
\right. .
\] 
In particular, $EK(n)$ is less than $3$ for all $n \leq 30$. 
On the other hand, 
it is easy to see that $G(K([45]))$ maps onto $G(K([3]))$, $G(K([5])$, $G(K([9]))$, and $G(K([15]))$. 
It follows that $EK(45) \geq 4$. 
\begin{problem}
Does there exist a $2$-bridge knot with less than $45$ crossings 
whose knot group maps onto $3$ (or $4$) distinct knot groups? 
In general, determine $EK(n)$ explicitly for all $n \geq 31$. 
\end{problem}

\section{generation function}

As shown in Example \ref{ex:ORS}, 
there exist $3$ distinct $2$-bridge knots with $9$ crossings whose knot groups admit 
epimorphisms onto the trefoil knot group. 
In this section, we generalize this result. 
Namely, for a given $2$-bridge knot $K(r)$,  
we determine the number of $2$-bridge knots $K(\tilde{r})$ 
which admit epimorphisms $\varphi : G(K(\tilde{r})) \to G(K(r))$, 
in terms of $c(K({\tilde r}))$.

\begin{theorem}\label{thm:generatingfunction}
For a given rational number $r$, 
we take the standard continued fraction expansion  $[a_1,a_2,\ldots,a_m]$ of $r$ 
and define the generating function as follows: 
\begin{enumerate}
\item if $(a_1,a_2,\ldots,a_m) \neq (a_m,\ldots,a_2,a_1)$, then 
\[
 f(r) = \sum_{n=1}^\infty \sum_{k=0}^\infty 2^{2n} \, 
\Big(
\begin{array}{c}
2n+ k - 1 \\ k
\end{array}
\Big)
\, 
t^{(2n+1) c(K(r)) + k} ,
\]
\item if $(a_1,a_2,\ldots,a_m) = (a_m,\ldots,a_2,a_1)$, then 
\[
 f(r) = \sum_{n=1}^\infty \sum_{k=0}^\infty \, g(n,k) \, t^{(2n+1) c(K(r)) + k}
\]
where 
\[
 g(n,k) = 
\left\{
\begin{array}{ll}
2^{2n-1} 
\Big(
\begin{array}{c}
2n+ k - 1 \\ k
\end{array}
\Big)
& k \mbox{ : odd} \\
2^{2n-1} 
\Big(
\begin{array}{c}
2n+ k - 1 \\ k
\end{array}
\Big)
+ 
2^{n-1} 
\Big(
\begin{array}{c}
n+ \frac{k}{2} - 1 \\ \frac{k}{2}
\end{array}
\Big)
& 
k \mbox{ : even} \\
\end{array}
\right. , 
\]
\end{enumerate}
and $\displaystyle{
\Big(
\begin{array}{c}
a \\ b
\end{array}
\Big) 
= \frac{a!}{b! \, (a-b)!}
}$. 
Then the number of $2$-bridge knots $K(\tilde{r})$ 
which admit epimorphisms $\varphi : G(K(\tilde{r})) \to G(K(r))$ 
is the coefficient of $t^{c(K(\tilde{r}))}$. 
\end{theorem}

\begin{proof}
We will count the number of $2$-bridge knots with $(2n+1) c(K(r)) + k$ crossings
which correspond to rational numbers 
\[
 \tilde{r} = 
[\varepsilon_1 {\bf a}, 2 c_1, 
\varepsilon_2 {\bf a}^{-1}, 2 c_2, 
\varepsilon_3 {\bf a}, 2 c_3, 
\varepsilon_4 {\bf a}^{-1}, 2 c_4, 
\ldots, 
\varepsilon_{2n} {\bf a}^{-1}, 2 c_{2n}, \varepsilon_{2n+1} {\bf a}] , 
\]
where ${\bf a} = (a_1,\ldots,a_m)$. 
The crossing number $c(K(r))$ is $\sum_{i=1}^m a_i$. 
Compared with Theorem \ref{thm:crossingnumber}, we have 
\[
 k = \sum_{i=1}^{2n} \bar{c}_i 
= \sum_{i=1}^{2n} 2 |c_i| - \psi (i) - \bar{\psi} (i)
\]
where $\bar{c}_i \geq 0$ by Proposition \ref{prop:ci}.

Suppose that $\bar{c}_i = j \, \, (\geq 0)$. 
Then $(\varepsilon_i {\bf a}^{\pm 1}, 2 c_i, \varepsilon_{i+1} {\bf a}^{\mp 1})$, 
which is a part of $\tilde{r}$, has 
the following possibilities:
\begin{enumerate}
\item if $j$ is even and $\varepsilon_i = 1$, then 
\[
(\varepsilon_i {\bf a}^{\pm 1}, 2 c_i, \varepsilon_{i+1} {\bf a}^{\mp 1})
= ({\bf a}^{\pm 1}, j, {\bf a}^{\mp 1}) \mbox{ or } ({\bf a}^{\pm 1}, -(j+2), {\bf a}^{\mp 1}),
\]
\item if $j$ is even and $\varepsilon_i = -1$, then 
\[
(\varepsilon_i {\bf a}^{\pm 1}, 2 c_i, \varepsilon_{i+1} {\bf a}^{\mp 1})
= (-{\bf a}^{\pm 1}, -j, -{\bf a}^{\mp 1}) \mbox{ or } (-{\bf a}^{\pm 1}, j+2,- {\bf a}^{\mp 1}),
\]
\item if $j$ is odd and $\varepsilon_i = 1$, then 
\[
(\varepsilon_i {\bf a}^{\pm 1}, 2 c_i, \varepsilon_{i+1} {\bf a}^{\mp 1})
= ({\bf a}^{\pm 1}, j+1, -{\bf a}^{\mp 1}) \mbox{ or } ({\bf a}^{\pm 1}, -(j+1), -{\bf a}^{\mp 1}),
\]
\item if $j$ is odd and $\varepsilon_i = -1$, then 
\[
(\varepsilon_i {\bf a}^{\pm 1}, 2 c_i, \varepsilon_{i+1} {\bf a}^{\mp 1})
= (-{\bf a}^{\pm 1}, j+1, {\bf a}^{\mp 1}) \mbox{ or } (-{\bf a}^{\pm 1}, -(j+1), {\bf a}^{\mp 1}).
\]
\end{enumerate}
Therefore $(\varepsilon_i {\bf a}^{\pm 1}, 2 c_i, \varepsilon_{i+1} {\bf a}^{\mp 1})$ 
always has $2$ possibilities. 
Besides, there are $\displaystyle{\Big(
\begin{array}{c}
2n+ k - 1 \\ k
\end{array}
\Big)}$ cases for $(\bar{c}_1,\ldots,\bar{c}_{2n})$, 
namely, 
\[
(\bar{c}_1,\bar{c}_2,\ldots,\bar{c}_{2n}) = 
(k,0,\ldots,0), (k-1,1,\ldots,0), \ldots, (0,0,\ldots,0,k).
\]
Hence the number of $2$-bridge knots with $(2n+1) c(K(r)) + k$ crossings 
is $\displaystyle{
2^{2n} 
\Big(
\begin{array}{c}
2n+ k - 1 \\ k
\end{array}
\Big)
}$ and we get the generating function of $(1)$. 
In the case of $(a_1,\ldots,a_m) = (a_m,\ldots,a_1)$, we see 
\[
K([\varepsilon_1 {\bf a}, 2 c_1, \varepsilon_2 {\bf a}^{-1}, 
\ldots, 
2 c_{2n}, \varepsilon_{2n+1} {\bf a}])
=
K([\varepsilon_{2n+1} {\bf a}, 2 c_{2n}, 
\ldots, 
\varepsilon_{2} {\bf a}^{-1}, 2 c_1, \varepsilon_1 {\bf a}]) .
\]
It implies that 
if $\tilde{r}$ is not symmetric, that is, if $\tilde{r}$ is not in the form 
\[
 [\varepsilon_1 {\bf a}, 2 c_1, \ldots, 
2 c_n, \varepsilon_{n+1} {\bf a}^{\pm 1}, 2 c_n, \ldots, 
2 c_1, \varepsilon_1 {\bf a}] , 
\]
we counted the same knot twice. 
Then the number of knots is 
\begin{align*}
& \frac{
2^{2n} 
\Big(
\begin{array}{c}
2n+ k - 1 \\ k
\end{array}
\Big) - 
2^{n} 
\Big(
\begin{array}{c}
n+ \frac{k}{2} - 1 \\ \frac{k}{2}
\end{array}
\Big)
}{2} + 
2^{n} 
\Big(
\begin{array}{c}
n+ \frac{k}{2} - 1 \\ \frac{k}{2}
\end{array}
\Big) \\
&= 
2^{2n-1} 
\Big(
\begin{array}{c}
2n+ k - 1 \\ k
\end{array}
\Big)
+
2^{n-1} 
\Big(
\begin{array}{c}
n+ \frac{k}{2} - 1 \\ \frac{k}{2}
\end{array}
\Big).
\end{align*}
Notice that if $k$ is odd, then $\tilde{r}$ must not be symmetric. 
As we saw in Section \ref{sect:cfe}, if the standard continued fraction expansions are not the same, 
then the $2$-bridge knots are not equivalent. 
We can get the standard fraction expansion of $\tilde{r}$ by Lemma \ref{lem:nonpositivecfe}. 
It shows that these knots which are obtained by the Ohtsuki-Riley-Sakuma construction are not equivalent. 
This completes the proof. 
\end{proof}

\begin{example}
First, we apply Theorem \ref{thm:generatingfunction} to the trefoil knot. 
The generating function for the trefoil $K(1/3) = K([3])$ is 
\begin{align*}
 f(1/3) =& \, 3 t^9 + 4 t^{10} + 7 t^{11} + 8 t^{12} + 11 t^{13} + 12 t^{14} + 25 t^{15} 
+ 48 t^{16} + 103 t^{17} + 180 t^{18} \\ 
& + 309 t^{19} + 472 t^{20} + 743 t^{21} + 1180 t^{22} + 2045 t^{23} + 3584 t^{24} + 6391 t^{25} + 
\cdots .
\end{align*}
Then the number of $2$-bridge knots with $9$ crossings whose knot groups 
admit epimorphisms onto the trefoil knot group is the coefficient of $t^9$, 
which is $3$. 
These $2$-bridge knots are $9_1,9_6,9_{23}$, as shown in Example \ref{ex:ORS}. 
Similarly, the knot groups of $4$ (respectively \ $7$) distinct $2$-bridge knots with $10$ 
(respectively $11$) 
crossings, 
which are $10_5,10_9,10_{32},10_{40}$ 
(respectively $11a_{117},11a_{175},11a_{176},11a_{203},11a_{236},11a_{306},11a_{355}$) 
as shown in \cite{kitano-suzuki1}
(respectively \cite{HKMS}), 
admit epimorphisms onto the trefoil knot group, and so on. 

Another example shows the generating function for $5_2 = K(3/7) = K([2,3])$: 
\begin{align*}
f(3/7) =& \, 4 t^{15} + 8 t^{16} + 12 t^{17} + 16 t^{18} + 20 t^{19} + 24 t^{20} + 28 t^{21} + 
 32 t^{22} + 36 t^{23} + 40 t^{24} \\
& + 60 t^{25} + 112 t^{26} + 212 t^{27} + 376 t^{28} + 620 t^{29} + 960 t^{30} + \cdots .
\end{align*}
\end{example}

\section*{Acknowledgments}
The author wishes to express his thanks to Makoto Sakuma and Takayuki Morifuji for helpful communications. 
He also thanks an anonymous referee for some useful comments. 
This work is partially supported by KAKENHI grant No.\ 16K05159 and 15H03618 from the Japan Society for the Promotion of Science.

\end{document}